\documentclass[11pt,reqno]{amsart}
\usepackage{amssymb}
\usepackage{fullpage}
\usepackage{amsthm}
\usepackage{bm}
\usepackage{latexsym}
\usepackage{float}
\restylefloat{table}
\usepackage{amsmath}
\usepackage{eufrak}
\usepackage{mathrsfs}
\usepackage[font={small,it}]{caption}
\usepackage{amscd}
\usepackage[all,cmtip]{xy}
\usepackage[usenames]{color}
\usepackage{graphicx}
\usepackage{color}
\usepackage{amscd}
\usepackage{float}
\usepackage{graphics}
\usepackage{tikz}
\usepackage{tikz-cd}
\usepackage{comment}
\usepackage{xspace}
\usepackage{mathtools}

\usepackage{amssymb}
\usepackage{amsthm}
\usepackage{graphicx}
\usepackage{subfig}
\usepackage{enumerate}
\usepackage{hyperref}

\usetikzlibrary{patterns,decorations.pathreplacing}
\usepackage{marginnote}
\theoremstyle{plain}

\newtheorem{theorem}{Theorem}[section]
\newtheorem{proposition}[theorem]{Proposition}
\newtheorem{lemma}[theorem]{Lemma}
\newtheorem{corollary}[theorem]{Corollary}
\newtheorem{conjecture}[theorem]{Conjecture}

\theoremstyle{definition}

\newcommand{\appsection}[1]{\let\oldthesection\thesection
\renewcommand{\thesection}{Appendix \oldthesection}
\section{#1}\let\thesection\oldthesection}
\newtheorem{definition}[theorem]{Definition}

\newtheorem{observation}[theorem]{Observation}
\theoremstyle{remark}

\newtheorem{remark}[theorem]{Remark}

\def\D{{\mathbb{D}}}

\def\Z{{\mathbb{Z}}}
\def\F{{\mathbb{F}}}
\def\Q{{\mathbb{Q}}}

\DeclareMathOperator{\rank}{rank}

\usepackage{microtype}

\DeclareMathOperator{\Sym}{Sym}

\date{\today.\\
\indent \textit{2010 Mathematics Subject Classification.} Primary 11G05; Secondary 11G18, 11G40.\\
\indent \textit{Key words and phrases.} Watkins's conjecture, elliptic curves, modular degree, Mordell-Weil rank.\\
\indent This research was supported by ANID Doctorado Nacional 21190304.}
 
\begin{document}
\title[Watkins's conjecture and Quadratic Twists]{Watkins's conjecture for quadratic twists of Elliptic Curves with Prime Power Conductor}

\author[Jerson Caro]{Jerson Caro}
\email{jlcaro@bu.edu}
\address{Department of Mathematics \& Statistics, Boston University, 665 Commonwealth Avenue, Boston, MA 02215, USA}

\maketitle

{\centering\footnotesize \textit{Dedicated to the memory of my father}.\par}
\begin{abstract}
Watkins's conjecture asserts that the rank of an elliptic curve is upper bounded by the $2$-adic valuation of its modular degree. We show that this conjecture is satisfied when $E$ is any quadratic twist of an elliptic curve with a rational point of order $2$ and prime power conductor, in particular, for the congruent number elliptic curves. Furthermore, we give a lower bound for the congruence number for elliptic curves of the form $y^2=x^3-dx$, with $d$ a fourth power free integer.
\end{abstract}

%\tableofcontents
%----------------------------------------------------------------
\section{Introduction} \label{intro}
\noindent 
Let $E$ be an elliptic curve defined over $\Q$. The modularity theorem \cite{wiles1995modular, taylor1995ring, breuil2001modularity} ensures the existence of a non-constant morphism $\phi: X_0 (N) \to E$ defined over $\Q$. Denote by $\phi_E$ the morphism, up to sign, which has minimal degree and which sends the cusp $i\infty$ to the neutral point of $E$.\textit{ The modular degree} $m_E$ of $E$ is the degree of $\phi_E$.
There are many relevant conjectures in number theory about this invariant. One of them, equivalent to ABC conjecture \cite{frey1989links}, is to give polynomial bounds of its size in terms of the conductor. The following conjecture is the main topic of this paper.
\begin{conjecture}[Watkins \cite{watkins2002computing}]
For every elliptic curve $E$ over $\Q$ we have $r \leq \nu_2(m_E)$, where $\nu_p$ denotes the $p$-adic valuation and $r\coloneqq\rank_\Z (E(\Q))$.
\end{conjecture}
One particular case of this conjecture is to prove that when $m_E$ is odd, then $\rank_\Z(E(\Q))=0$. In this direction, there is great progress. For example, Kazalicki and Kohen \cite{kazalicki2018special},\cite{kazalicki2019corrigendum} proved that if the congruence number $\delta_E$ of $E$ (which is a multiple of the modular degree $m_E$) is odd, then $\rank_\Z(E(\Q))=0$. From this, it may be deduced (see Proposition \ref{wconj prime}) that Watkins's conjecture holds for elliptic curves with conductor a prime power and nontrivial rational $2$-torsion, (cf. Section \ref{sec 4}). 

When the elliptic curves has a rational point of order $2$, using Selmer groups an upper bound for $r$ can be given in terms of $\omega(N)$. Here, $N$ is the conductor of $E$, and $\omega(N)$ is the number of distinct prime factors of $N$ (cf. Section \ref{Mordell-Weil Rank}). This upper bound allowed Esparza-Lozano \& Pasten \cite{esparza2021conjecture} to prove that Watkins's conjecture holds for a quadratic twist $E^{(D)}$ by $D$ with $\omega(D)$ large enough, whenever $E$ has a rational point of order $2$. 

A natural question emerges: When do all the quadratic twists of an elliptic curve $E$ satisfy Watkins's conjecture? In this paper we prove that this conjecture holds for any quadratic twist of an elliptic curve with a rational point of order $2$ and prime power conductor.
\begin{theorem}\label{main theorem}
Let $E$ be an elliptic curve with a rational point of order $2$. Assume that $E$ is a quadratic twist of an elliptic curve with prime power conductor. Then $E$ satisfies Watkins's conjecture.
\end{theorem}
To prove that if $\rank_\Z(E(\Q))>0$ then $2\mid m_E$, because of Theorem 1 in \cite{calegari2009elliptic}, the only missing case is when $E$ has conductor $N$ divisible by at most two odd primes, additive reduction at $2$ and a rational point of order $2$. Thus, the previous Theorem covers many of those cases. 

Another application of this theorem is the following: An integer $D$ is a \textit{congruent number} if there exists a right triangle such that its sides are rational and its area equals $D$. An important problem in number theory is to know when an integer number $D$ is a congruent number, which is equivalent to know if the elliptic curve $y^2=x^3-D^2x$ has positive Mordell-Weil rank $r$, see \cite{koblitz2012introduction}. A crucial observation is that for an integer $D$, $E^{(D)}:y^2=x^3-D^2x$ is the quadratic twist by $D$ of the elliptic curve $E: y^2=x^3-x$, which has conductor $32$. Applying  \cite[Theorem 1.2]{esparza2021conjecture}, we have that for a positive squarefree integer $D$, with $\omega(D)\geq 12$, Watkins's conjecture holds for $E^{(D)}$. Thanks to Theorem \ref{main theorem}, we show this conjecture unconditionally. More precisely:
\begin{corollary}
Watkins's conjecture holds for all congruent curves $E^{(D)}:y^2=x^3-D^2x$.
\end{corollary}
Since the elliptic curve $y^2=x^3-x$ has complex multiplication by $\Z[i]$ we may consider its quartic twists. However, the process we use to prove Theorem \ref{main theorem} seems not useful for quartic twists, since we need to find a lower bound of the $2$-adic valuation of an infinite product. Although we cannot give an applicable lower bound for $\nu_2(m_E)$, we found an alternative process that provides us a lower bound of $\nu_2(\delta_E)$.
\begin{theorem}\label{2main theorem}
Let $d$ be an odd squarefree integer and $D$ any divisor of $d$. For the elliptic curve $E:y^2=x^3-dD^2x$ we have that  
\[
2\left\lfloor\frac{\omega(d)+1}{2}\right\rfloor+1\leq \nu_2(\delta_E).
\]
\end{theorem}
This theorem allows us to prove that for some elliptic curves $\rank_\Z (E(\Q))\leq \nu_2(\delta_E)$. This result goes into the direction of Watkins's conjecture since $m_E\mid \delta_E$ as we mentioned before.
\begin{corollary}\label{corollary 2main}
Let $p$ be an odd prime. Then for $E$ an elliptic curve $y^2=x^3-px$ or $y^2=x^3-p^3x$, we have that $\rank_\Z (E(\Q)) <\nu_2(\delta_E)$.
\end{corollary}

\section*{Acknowledgment}
I would like to thank Hector Pasten for suggesting this
problem to me, carefully reading the preliminary version of this manuscript and numerous helpful remarks. I thank the anonymous referees for several valuable comments on an earlier version of this manuscript. This paper is based on my thesis, thus, I am very
grateful to my examiners Ricardo Menares and Fabien Pazuki for several suggestions
which have improved the present work. This research was supported by ANID Doctorado Nacional 21190304.

\section{Preliminaries}\label{results}

\subsection{$2$-adic valuation of the modular degree}
\cite[Lemma 3.1]{esparza2021conjecture} gives an equation that relates the modular degree of an elliptic curve $E$ with that of its quadratic twist $E^{(D)}$ by $D$, where $D$ is any fundamental discriminant. We denote by $N$ and $N^{(D)}$ the conductors  of $E$ and $E^{(D)}$, respectively. Before showing this equation, we have to define some invariants which appear on it.
\subsubsection{Petersson Norm} Let $S_2(\Gamma_0 (N ))$ be the space of weight $2$ cuspidal holomorphic modular forms; over this space, we have an inner product that allows us to define the following norm.
\begin{definition}
The Petersson norm of $f\in S_2(\Gamma_0 (N ))$ is defined by
\[
\|f\|_N=\left(\int_{\Gamma_0(N)\backslash \mathfrak{h}}|f(z)|^2 dx \wedge dy\right)^{1/2}, \qquad z = x + iy \text{ and }y>0.
\]
\end{definition}
\begin{observation}
Although this definition depends on the level $N$, we know that if $N\mid M$ and $f\in S_2(\Gamma_0 (N ))$, then $f\in S_2(\Gamma_0 (M))$ and $\|f\|^2_M = [\Gamma_0 (N ) : \Gamma_0 (M )]  \|f\|^2_N$.
\end{observation}
\subsubsection{Manin Constant} Let $E$ be an elliptic curve defined over $\Q$ of conductor $N$ and let $\omega_E$ be its N\'eron differential. We have that $\phi^{*}_{E} \omega_E$ is a
regular differential on $X_0 (N )$, which implies the following formula:
\[
\phi^{*}_{E} \omega_E = 2\pi ic_E f_E(z)dz
\]
where $c_E$ is a rational number and $f_E$ denotes the Hecke newform attached to $E$. Due to \cite[Proposition 2]{edixhoven1991manin}, $c_E$, which we call the Manin constant, is an integer uniquely defined up to sign.

The mentioned equation given by \cite[Lemma 3.1] {esparza2021conjecture} is the following:
\begin{equation}\label{twist-modular degree}
\frac{m_{E^{(D)}}}{c_{E^{(D)}}^2}=\frac{m_E}{c_E^2}\times\frac{\|f_{E^{(D)}}\|^2_{N^{(D)}}}{\|f_E\|^2_N}\left|\frac{\Delta_{E^{(D)}}}{\Delta_E}\right|^{1/6},
\end{equation}
where $\Delta_E$ denotes the global minimal discriminant of $E$. Equation \eqref{twist-modular degree} implies the following lemma:
\begin{lemma}\label{valuation modular}
Let $E$ be an elliptic curve and $E^{(D)}$ its quadratic twist by $D$. Then 
\begin{equation}\label{val}
\nu_2(m_{E^{(D)}})\geq\nu_2\left(\frac{m_E}{c_E^2}\right)+\nu_2\left(\frac{\|f_{E^{(D)}}\|^2_{N^{(D)}}}{\|f_E\|^2_N}\right)+\frac{1}{6}\nu_2\left(\frac{\Delta_{E^{(D)}}}{\Delta_E}\right).
\end{equation}
\end{lemma}
\subsection{Elliptic curves with a rational point of order 2 and prime power conductor}\label{classification}
In this section, we classify all the elliptic curves defined over $\Q$ with nontrivial rational $2$-torsion and conductor a power of a prime. 
\begin{remark}\label{conductor}
Suppose that $E$ is an elliptic curve with conductor a power of a prime $p^\alpha$. Then we have that $\alpha\leq 2$ for $p>3$, $\alpha\leq 5$ for $p=3$ and $\alpha\leq 8$ for $p=2$ (see \cite[IV.10]{silverman1994}).
\end{remark}
We shall write $[a_1,a_2,a_3,a_4,a_6]$ for the elliptic curve
\[
y^2+a_1xy+a_3y=x^3+a_2x^2+a_4x+a_6.
\]
We begin with the elliptic curves with prime conductor. Setzer \cite{setzer1975elliptic} proved that for $p\neq 17$ there exists an elliptic curve with prime conductor $p$ and a rational point of $2$ if and only if $p=u^2+64$ for some integer $u$, in which case there are two nonisomorphic elliptic curves with conductor $p$. The minimal models of these elliptic curves are
\begin{table}[H]
    \centering
    \begin{tabular}{|c|c|c|c|}
 \hline
LMFDB label & Weierstrass coefficients & $\Delta$  & $2$-Torsion\\ 
 \hline
 $p.a1$ & $[1, (u-1)/4, 0, -1, 0]$ & $p$ & $\Z/2\Z$\\
 $p.a2$ & $[1,(u-1)/4,0,4,u]$ & $-p^2$ & $\Z/2\Z$\\
  \hline
\end{tabular}
    \caption{Elliptic curves with prime conductor $p>17$ and nontrivial rational $2$-torsion}
    \label{table1}
\end{table}
A $2$-isogeny connects these two elliptic curves. Moreover, the work of Mestre \& Oesterl\'e \cite{mestre1989courbes} implies that the curve $p.a2$ is the $X_0(p)$-optimal, where an elliptic curve $E$ is called  $X_0(N)$-optimal if it has the minimal modular degree $m_E$ in its isogeny class. Additionally, every other modular parametrization of a curve in the isogeny class factors through the minimal-degree modular
parametrization of the optimal curve. 

For $p=17$ Setzer \cite{setzer1975elliptic} shows that there are four nonisomorphic elliptic curves. From now on, $(*)$ indicates which elliptic curve is the $X_0(N)$-optimal. 
\begin{table}[H]
    \centering
\begin{tabular}{|c|c|c|c|c|c|}
 \hline
LMFDB label & Weierstrass coefficients & $m_E$ & $c_E$ & $\Delta$ & $2$-Torsion \\ 
 \hline
 $17.a1$ & $[1, -1, 1, -91, -310]$ & $4$ & $2$ & $17$ & $\Z/2\Z$ \\
 $17.a2$ & $[1,-1,1,-6,-4]$ & $2$ & $2$ & $17^2$ & $(\Z/2\Z)^2$ \\
 $17.a3(*)$ & $[1,-1,1,-1,-14]$ & $1$ & $1$ & $-17^4$ & $\Z/2\Z$\\
 $17.a4$ & $[1, -1, 1, -1, 0]$ & $4$ & $4$ & $17$& $\Z/2\Z$\\
 \hline
\end{tabular}
\caption{Elliptic curves with conductor $17$}
    \label{table2}
\end{table}
On the other hand, Mulholland \cite{mulholland2006elliptic} proved that for $p>3$ the elliptic curves with a rational point of order $2$ and conductor $p^2$ are the quadratic twist by $p$ of the elliptic curves in Tables \ref{table1} and \ref{table2} by $p$, together with the ones with conductor $49$ listed below
\begin{table}[H]
    \centering
\begin{tabular}{|c|c|c|c|c|c|}
 \hline
LMFDB label & Weierstrass coefficients & $m_E$ & $c_E$ & $\Delta$ & $2$-Torsion  \\ 
 \hline
 $49.a1$ & $[1, -1, 0, -1822, 30393]$ & $14$ & $1$ & $7^9$ & $\Z/2\Z$ \\
 $49.a2$ & $[1,-1,0,-107,552]$ & $7$ & $1$ & $7^9$ & $\Z/2\Z$\\
 $49.a3$ & $[1, -1, 0, -37, -78]$ & $2$ & $1$ & $7^3$ & $\Z/2\Z$\\
 $49.a4(*)$ & $[1,-1,0,-2,-1]$ & $1$ & $1$ & $7^2$ & $\Z/2\Z$\\
 \hline
\end{tabular}
\caption{Elliptic curves with conductor $49$}
    \label{table3}
\end{table}
By Remark \ref{conductor} we have classified the elliptic curves whose conductor is a power of a prime for $p>3$ and nontrivial rational $2$-torsion. Using the database \cite{LMFDB}, we can classify the elliptic curves with a rational point of order $2$ and conductor a power of $2$ or $3$. We noticed that there are no elliptic curves with a rational point of order $2$ of conductor $3^m$  for any integer $m$. Thus, by Remark \ref{conductor} we only need to list the ones with conductor $2^m$ with $m\in\{5,6,7,8\}$.  

The following table shows the elliptic curves with conductor $2^5$
\begin{table}[H]
    \centering
\begin{tabular}{|c|c|c|c|c|c|}
 \hline
LMFDB label & Weierstrass coefficients & $m_E$ & $c_E$ & $\Delta$ & $2$-Torsion \\ 
 \hline
 $32.a1$ & $[0,0,0,-11,-14]$ & $4$ & $2$ & $2^9$ & $\Z/2\Z$ \\
 $32.a2$ & $[0,0,0,-11,14]$ & $4$ & $2$ & $2^9$ & $\Z/2\Z$\\
 $32.a3$ & $[0,0,0,-1,0]$ & $2$ & $2$ & $2^6$ & $(\Z/2\Z)^2$\\
 $32.a4(*)$ & $[0,0,0,4,0]$ & $1$ & $1$ & $-2^{12}$ & $\Z/2\Z$\\
 \hline
\end{tabular}
\caption{Elliptic curves with conductor $32$}
    \label{table4}
\end{table}
Let us mention that 32.a3 is called the \textit{congruent number} curve. Meanwhile, the elliptic curves of conductor $2^6$ are the quadratic twists of the previous ones by $2$. Finally, the elliptic curves with conductor $2^7$ are listed in the following table
\begin{table}[H]
    \centering
\begin{tabular}{|c|c|c|c|c|c|}
 \hline
 LMFDB label & Weierstrass coefficients & $m_E$ & $c_E$ &$\Delta$  & $2$-Torsion
 \\ 
 \hline
$128.a1$ & $[0, 1, 0, -9, 7]$ & $8$ & $1$ & $2^{13}$& $\Z/2\Z$\\
$128.a2(*)$ & $[0, 1, 0, 1, 1]$ & $4$ & $1$ & $-2^8$ & $\Z/2\Z$\\
$128.b1$ & $[0, 1, 0, -2, -2]$ & $16$ & $2$ & $2^7$ & $\Z/2\Z$\\
$128.b2(*)$ & $[0, 1, 0, 3, -5]$ & $8$ & $1$ & $-2^{14}$ & $\Z/2\Z$\\
$128.c1$ & $[0, -1, 0, -9, -7]$ & $8$ & $1$ & $2^{13}$ & $\Z/2\Z$\\
$128.c2(*)$ & $[0, -1, 0, 1, -1]$ & $4$ & $1$ & $-2^8$ & $\Z/2\Z$\\
$128.d1$ & $[0, -1, 0, -2, 2]$ & $16$ & $2$ & $2^7$ & $\Z/2\Z$\\
$128.d2(*)$ & $[0, -1, 0, 3, 5]$ & $8$ & $1$ & $- 2^{14}$ & $\Z/2\Z$\\
 \hline
\end{tabular}
\caption{Elliptic curves with conductor $128$}
    \label{table5}
\end{table}
%%%%%%%%%%%%%%%%%%%%%%
and again the elliptic curves of conductor $2^8$ are the quadratic twists of the previous ones by $2$.
\subsection{The Mordell-Weil Rank}\label{Mordell-Weil Rank} \cite[Section X.4]{silverman2009arithmetic}  gives a bound for a 2-isogeny Selmer rank. This work allows Caro \& Pasten \cite{caro2021watkins} to find an upper bound for the Mordell-Weil rank of an elliptic curve with nontrivial rational $2$-torsion
\begin{equation}\label{rank}
\rank_\Z(E(\Q)) \leq 2\omega(N) - 1.    
\end{equation}
Furthermore, \cite[Proposition 1.1]{aguirre2010elliptic} shows that if the elliptic curve has minimal Weierstrass equation $y^2=x^3+Ax^2+Bx$, we obtain that
\begin{equation}\label{rank32}
\rank_\Z(E(\Q)) \leq \omega(A^2-4B) + \omega(B) - 1.    
\end{equation}
We define the function $\delta_2(D)$ on the set of integers as follows:
\[
\delta_2(D)=
\begin{cases}
1 & \text{if }2\mid D\\
0 & \text{otherwise.}
\end{cases}
\]
Using inequality \eqref{rank32} we have the following lemma:
\begin{lemma}\label{rankbound}
Let $E$ be an elliptic curve with a rational point of order $2$ and prime power conductor $N=p^\alpha$, and let $E^{(D)}$ its quadratic twist by $D$, with $D$ a fundamental discriminant. Then we have
\begin{equation*}
\rank_\Z(E^{(D)}(\Q)) \leq 2\omega(D) + 1-2\nu_p(D).    
\end{equation*}
Even sharper, if $E$ is $32.a3$ we have
\begin{equation*}
\rank_\Z(E^{(D)}(\Q)) \leq 2\omega(D) -\delta_2(D).
\end{equation*}
\end{lemma}
\begin{proof}
We know that $E^{(D)}[2]\cong E[2]$ as Galois modules because one is the other multiplied by the character $\chi_D$ and $\chi_D$ takes values congruent to $1$ (mod $2$). Then $E^{(D)}$ also has nontrivial rational $2$-torsion, so applying the inequality \eqref{rank} we have
\begin{align*}
\rank_\Z(E^{(D)}(\Q)) &\leq 2\omega(N^{(D)})-1\leq 2(\omega(D) + (1-\nu_p(D)))-1\\ &=2\omega(D) + 1-2\nu_p(D).    
\end{align*}
The second inequality is derived from the observation that the primes dividing $N^{(D)}$ are are exactly those primes that divide $D$ and $p$. If $E$ is $32.a3$ its quadratic twist $E^{(D)}$ is $y^2=x^3-D^2x$, in particular, we can apply inequality \eqref{rank32}, so we obtain
\begin{align*}
\rank_\Z(E(\Q)) &\leq \omega(4D^2) + \omega(-D^2) - 1\\
&=\omega(D)+(1-\delta_2(D)) + \omega(D) - 1\\
&=2\omega(D)-\delta_2(D),
\end{align*}
which ends the proof.
\end{proof}
\begin{remark}\label{remark rank}
Note that we can also apply the inequality (\ref{rank32}) to $y^2=x^3-dx$, for any integer $d$, and again, we obtain that $\rank_\Z(E(\Q))\leq 2\omega(d)-\delta_2(d)$. 
\end{remark}
\section{Lower bounds for some $2$-adic valuations}
This section aims to give lower bounds for the $2$-adic valuation of the invariants in Lemma \ref{valuation modular}. 

\subsection{Minimal discriminants}
We start by finding a lower bound for 
\begin{equation}\label{valuation discriminant}
\nu_2\left(\left(\frac{\Delta_{E^{(D)}}}{\Delta_E}\right)^{1/6}\right)=\frac{1}{6}\nu_2\left(\frac{\Delta_{E^{(D)}}}{\Delta_E}\right).
\end{equation}
\begin{definition}
Let $p$ be a prime. The $p$-adic signature of an elliptic curve $E$ is the triple $(\nu_p(c_4(E)), \nu_p(c_6(E)), \nu_p(\Delta(E)))$, where $c_4,c_6$ are the usual Weierstrass invariants of a minimal model of $E$ and $\Delta(E)$ denotes the minimal discriminant of $E$.
\end{definition}
Pal \cite{pal2012periods} classifies the valuation \eqref{valuation discriminant} in terms of the $2$-adic signature of $E$. To begin with, we compute the $2$-adic signature of an elliptic curve with odd discriminant and a rational point of order $2$.
\begin{lemma}\label{Lemma signature}
Let $E$ be an elliptic curve with a rational point of order $2$ and odd discriminant. Then the $2$-adic signature of $E$ is $(0,0,0)$. 
\end{lemma}
\begin{proof}
Assume that the minimal model of $E$ is of the form
\[
y^2+a_1xy+a_3y=x^3+a_2x^2+a_4x+a_6.
\]
Since $E$ has good reduction at $2$, then either $a_1$ or $a_3$ must be odd. Furthermore, because $E(\Q)[2]\neq \{0\}$, there exists $x_0\in \Q$ such that
\begin{equation}\label{eq1}
x_0^3+b_2x_0^2+8b_4x_0+16b_6=0,    
\end{equation}
where $b_2$, $b_4$ and $b_6$ are the usual Weierstrass invariants. We first suppose for a contradiction that $a_1$ is even. Then $\nu_2(b_2)=\nu_2(a_1^2+4a_2)\geq 2$,  $\nu_2(b_4)=\nu_2(a_1a_3+2a_4)\geq 1$, and  $a_3$ must be odd, hence $b_6=a_3^2+4a_6$ is odd too. Consequently, the Newton polygon (as it is noticed in Section 2 of \cite{setzer1975elliptic}) attached to \eqref{eq1} is a line with slope $-4/3$, so,  by Dumas’s irreducibility criterion (cf. the corollary in p.55 of \cite{prasolov2004polynomials}) this polynomial has no rational solutions, which is a contradiction. Hence $b_2$ is odd, and therefore $c_4 = b_2^2 -24b_4$ and $c_6 =-b_2^3 +36b_2b_4-216b_6$ are odd too.
\end{proof}

\begin{corollary}\label{odd discriminant}
Let $E$ be an elliptic curve with a rational point of order $2$ and odd discriminant. Then we have 
\[
\frac{1}{6}\nu_2\left(\frac{\Delta_{E^{(D)}}}{\Delta_E}\right)\geq \begin{cases}
0& \text{if }  D\equiv 1\text{(mod }4)\\
2& \text{if }  2\mid D.
\end{cases}
\]
\end{corollary}
\begin{proof}
By \cite[Proposition 2.4.2.(a)]{pal2012periods} if $D\equiv 1\text{ (mod }4)$, $\nu_2(\Delta_{E})=\nu_2(\Delta_{E^{(D)}})$. On the other hand, if $2\mid D$ the associated squarefree integer $D^{*}$ to $D$ is congruent to $-1$ or $2$ modulo $4$. Lemma \ref{Lemma signature} implies that the $2$-adic signatue of $E$ is $(0,0,0)$. By \cite[Proposition 2.4.2.(b).i]{pal2012periods} when $D^{*}\equiv -1\text{ (mod }4)$, $\nu_2(\Delta_{E})=\nu_2(\Delta_{E^{(D)}})=12$. Finally, by \cite[Proposition 2.4.2.(c).i]{pal2012periods} when $D^{*}\equiv 2\text{ (mod }4)$, $\nu_2(\Delta_{E})=\nu_2(\Delta_{E^{(D)}})=18$.
\end{proof}

Finally, we list the $2$-adic signature of the elliptic curves with conductor $2^5$ and $2^7$.
\begin{table}[H]
\centering
\begin{tabular}{|c|c| c|}
 \hline
 LMFDB label & $(c_4(E),c_6(E))$ & $2$-adic signature 
 \\ 
 \hline
$32.a1$ & $(528,12096)$ & $(4,6,9)$ \\
$32.a2$ & $(528,-12096)$ & $(4,6,9)$ \\
$32.a3$ & $(48,0)$ & $(4,\infty,6)$ \\
$32.a4$ & $(-192,0)$ & $(6,\infty,12)$\\
$128.a1$ & $(448,-8704)$ & $(6,9,13)$\\
$128.a2$ & $(-32,-640)$ & $(5,7,8)$\\
$128.b1$ & $(112,1088)$ & $(4,6,7)$\\
$128.b2$ & $(-128,5120)$ & $(7,10,14)$\\
$128.c1$ & $(448,3392)$ & $(6,6,13)$\\
$128.c2$ & $(-32,1088)$ & $(5,6,8)$\\
$128.d1$ & $(112,-2368)$ & $(4,6,7)$\\
$128.d2$ & $(-128,-3520)$ & $(7,6,14)$\\
 \hline
\end{tabular}
\caption{$2$-adic signature}
    \label{table6}
\end{table}
Analogously to the proof of Corollary \ref{odd discriminant}, Table \ref{table6} and \cite[Proposition 2.4.2]{pal2012periods} imply the following lemma:
\begin{lemma}\label{even discriminant}
Let $E$ be an elliptic curve with conductor $2^5$ or $2^7$. Then we have \[
\frac{1}{6}\nu_2\left(\frac{\Delta_{E^{(D)}}}{\Delta_E}\right)\geq-\delta_2(D)
\]
\end{lemma}
\subsection{Petersson norms}
Now we want to relate the $2$-adic valuation of the Petersson norms of $f_{E^{(D)}}$ and $f_E$.
\begin{proposition}\label{Petersson Norm}
Let $E$ be an elliptic curve with a rational point of order $2$ and minimal conductor $N$ among all its quadratic twists.
\begin{itemize}
    \item [(I)] Assume that $N$ is a power of an odd prime, and if $N=p^2$, we only consider $D$ such that $p\nmid D$. Then, we have 
    \[
    \nu_2\left(\frac{\|f_{E^{(D)}}\|^2_{N^{(D)}}}{\|f_E\|^2_N}\right)\geq 3\omega(D).
    \]
    \item [(II)] Furthermore, if $E$ is $17.a4$ in Table \ref{table2} we have
    \[
    \nu_2\left(\frac{\|f_{E^{(D)}}\|^2_{N^{(D)}}}{\|f_E\|^2_N}\right)\geq 4\omega(D).
    \]
    \item [(III)] If $N$ is $2^5$ or $2^7$ we have
    \[
\nu_2\left(\frac{\|f_{E^{(D)}}\|^2_{N^{(D)}}}{\|f_E\|^2_N}\right)\geq 3\omega(D)-2\delta_2(D).
\]
\item [(IV)] Furthermore, if $E$ is $32.a3$ in Table \ref{table4} we have
    \[
    \nu_2\left(\frac{\|f_{E^{(D)}}\|^2_{N^{(D)}}}{\|f_E\|^2_N}\right)\geq 4\omega(D)-3\delta_2(D).
    \]
\end{itemize}
\end{proposition}
\begin{proof}
The Rankin--Selberg method (cf. \cite{Shimura}) shows that
\[
\|f_E\|^2_N=\frac{N}{8\pi^3}L(\Sym f,2),
\]
where $L(\Sym f,2)$ denotes the symmetric square $L$-function associated to $f$ (see \cite[Equation 2]{delaunay2003computing} for its definition). Delaunay \cite{delaunay2003computing} compares the local factors of $L(\Sym f,2)$ and $L(\Sym f^{(D)},2)$ gives an equation relating $\|f_E\|^2_N$ to $\|f_{E^{(D)}}\|^2_{N^{(D)}}$ (see \cite[Theorem 1]{delaunay2003computing}). To state this equation, let us fix some notation.
For a prime number $q$ we define
$V(q)=(q-1)(q+1-a_q)(q+1+a_q)$, and $U(q)=(q-1)(q+1)$, where $a_q$ is the Fourier coefficient of $f_E$. Finally, we define $U_2=2(3-a_2)(3+a_2)$. 
Assume that $N$ is a power of an odd prime $p$. Since $p\mid D$ only if $N=p$, then according to the notation of Delaunay in \cite{delaunay2003computing}, $D_1=p$. Thus, taking $2$-adic valuations \cite[Theorem 1]{delaunay2003computing} implies that
\[
\nu_2\left(\frac{\|f_{E^{(D)}}\|^2_{N^{(D)}}}{\|f_E\|^2_N}\right)\geq\nu_p(D)\nu_2(U(p))+\nu_2(D)\nu_2(U_2)+\sum_{\substack{q\mid D\\ q\neq 2,p}}\nu_2(V(q)).
\]
In view of the fact that $E(\Q)[2]$ reduces injectively into $E(\F_q)$ for $q\notin \{2,p\}$, we have
$q + 1 \equiv a_q(E)$ (mod $2$), in particular, $\nu_2(V(q))\geq 3$. Furthermore, the database \cite{LMFDB} tells us that $2\mid 3-a_2,3+a_2$ for elliptic curves with conductor $17$ or $49$, and an inspection of the reduction modulo $2$ of $p.a1$ and $p.a2$ for $p>17$ shows that $2\mid\#E(\F_2)$, which shows that $\nu_2(U_2)\geq 3$. Finally, it is clear that $\nu_2(U(p))\geq 3$, then 
\[
\nu_2\left(\frac{\|f_{E^{(D)}}\|^2_{N^{(D)}}}{\|f_E\|^2_N}\right)\geq 3\left(\nu_p(D)+\delta_2(D)+\sum_{\substack{q\mid D\\ q\neq 2,p}}1\right)=3\omega(D),
\]
which proves (I).

To prove (II), we notice that $\#E(\Q)[4]=4$ and due to the fact that $E(\Q)[4]$ reduces injectively into $E(\F_q)$ for $q\notin \{2,p\}$, we have
$q + 1 \equiv a_q(E)$ (mod $4)$, in particular, $\nu_2(V(q))\geq 4$. We also know by the database \cite{LMFDB} that $a_2=-1$ then $\nu_2(U_2)=4$ and $\nu_2(U(17))=5$. Putting all together we obtain
\[
\nu_2\left(\frac{\|f_{E^{(D)}}\|^2_{N^{(D)}}}{\|f_E\|^2_N}\right)\geq 4\left(\nu_p(D)+\delta_2(D)+\sum_{\substack{q\mid D\\ q\neq 2,p}}1\right)=4\omega(D).
\]
For (III), we note that for $N$ equal to $2^5$ or $2^7$ \cite[Theorem 1]{delaunay2003computing} says that
\[
\nu_2\left(\frac{\|f_{E^{(D)}}\|^2_{N^{(D)}}}{\|f_E\|^2_N}\right)=\delta_2(D)+\sum_{\substack{q\mid D\\ q\neq 2}}\nu_2(V(q)).
\]
As we saw before $\nu_2(V(q))\geq 3$ for $q\neq 2$, then 
\[
\nu_2\left(\frac{\|f_{E^{(D)}}\|^2_{N^{(D)}}}{\|f_E\|^2_N}\right)\geq\nu_2(D)+3(\omega(D)-\delta_2(D))=3\omega(D)-2\delta_2(D).
\]
Finally, to prove (IV), we only need to notice that for the curve $32.a3$ $\#E(\Q)[2]=4$, then $\nu_2(V(q))\geq 4$ for $q\neq 2$,
which ends the proof.
\end{proof}
\begin{remark}\label{D=1(mod 4)}
Notice that if $D$ is prime and $D\equiv 1\text{ (mod }4)$ and relatively prime to $N$ we can improve the bounds given in (I) and (II), since $\nu_2(V(D))$ is higher. In this case, we have $\nu_2(\|f_{E^{(D)}}\|^2_{N^{(D)}}/\|f_E\|^2_N) \geq 4$ if $N$ is a power of an odd prime, and $\nu_2(\|f_{E^{(D)}}\|^2_{N^{(D)}}/\|f_E\|^2_N)\geq 5$ if $E$ is $17.a4$.
\end{remark}

\subsection{Manin constant}
Now, we compute $c_E$ of the elliptic curves $E$ in Table \ref{table1}.

\begin{proposition}\label{Manin}
Let $E$ be an elliptic curve, $X_0(N)$-optimal with odd squarefree conductor $N$, and let $E'$ an elliptic curve connected with $E$ by a $2$ isogeny $\theta:E\to E'$. Then we have that $c_E=1$ and $c_{E'}\in\{1,2\}$. 
\end{proposition}
\begin{proof}
By \cite[Corollary 4.2]{mazur1978rational}, $c_E$ must be a power of $2$, and  \cite[Theorem A]{abbes1996propos} says that if $p\mid c_E$ then $p\mid N$, which implies that $c_E=1$. Now, let $\mathcal{E}$ and $\mathcal{E}'$ be the Néron models of $E$ and $E'$, and $\omega,\omega'$ their respective Néron differentials. Since $\theta$ and $\theta^\vee$ define morphisms 
$\theta^*:H^0(\mathcal{E}', \Omega^1_{\mathcal{E}'/\Z})\to H^0(\mathcal{E}, \Omega^1_{\mathcal{E}/\Z})$ and $(\theta^\vee)^*:H^0(\mathcal{E}, \Omega^1_{\mathcal{E}/\Z})\to H^0(\mathcal{E}', \Omega^1_{\mathcal{E}'/\Z})$,
there are $a,b\in \Z$ such that $\theta^*\omega'=a\omega$ and $(\theta^\vee)^*\omega=b\omega'$. Due to the optimality of $\phi_E$ we have that $\phi_{E'}=\theta\circ\phi_E$, which implies that 
\[
\phi_{E'}^*\omega'=\phi_E^*\theta^*\omega'=a\phi_E^*\omega=2\pi i af_E(z)dz,
\]
hence $c_{E'}=a$. On the other hand, we have that $a\mid 2$ since 
\[
ab\omega=\theta^*(\theta^\vee)^*\omega=[2]^*\omega=2\omega,
\]
which ends the proof.
\end{proof}
\begin{corollary}\label{m/c2}
Let $E$ be an elliptic curve with a rational point of order $2$ and prime conductor $p>17$. Then $\nu_2(m_E/c^2_E)\geq -1$.
\end{corollary}
\begin{proof}
We denote $E_{p,1}$ and $E_{p,2}$ the curves $p.a1$ and $p.a2$ in Table \ref{table1}, respectively.  As we discussed before, there is a $2$-isogeny $\theta$ between these two curves, and the work of \cite{mestre1989courbes} shows that $E_{p,2}$ is $X_0(p)$-optimal. Because of Proposition \ref{Manin} $c_{E_{p,1}}\in\{1,2\}$ and $c_{E_{p,2}}=1$, in particular, $\nu_2(m_{E_{p,2}}/c^2_{E_{p,2}})\geq 0$. 
Since $\phi_{E_{p,1}}=\theta \circ\phi_{E_{p,2}}$, we have that $m_{E_{p,1}}=2m_{E_{p,2}}$ therefore $\nu_2(m_{E_{p,1}})\geq 1$, consequently
\[
\nu_2(m_{E_{p,1}}/c^2_{E_{p,1}})\geq \nu_2(m_{E_{p,1}})-2\geq -1,
\]
which gives the desired result.
\end{proof}
\section{The Main Result}\label{sec 4}

Before proving Theorem \ref{main theorem}, we need the following definition and proposition:
\begin{definition}
Let $E$ be an elliptic curve defined over $\Q$ and $f_E=\sum_{n=1}^\infty b_nq^n\in S_2(\Gamma_0(N))$ be its Hecke newform. The congruence number $\delta_E$ of $E$ is the largest integer such that there is a modular form $g=\sum_{n=1}^\infty b_nq^n \in S_2(\Gamma_0(N))$ with $b_n\in\Z$, such that $g$ and $f_E$ are orthogonal with respect to the Petersson inner product, and $a_n \equiv b_n \text{ (mod } \delta_E)$ for all $n$.
\end{definition}
\begin{remark} There are some relations between the modular degree and the congruent number. The most relevant is $m_E\mid \delta_E$, whenever $E$ is $X_0(N)$-optimal \cite{cojocaru2004modular}.
\end{remark}
The following proposition will play a key role in the proof of Theorem \ref{main theorem} when the prime $p$, dividing both the conductor and the fundamental discriminant $D$, is considered.
\begin{proposition}\label{wconj prime}
Watkins's conjecture holds for every elliptic curve $E$ of prime power conductor and a rational point of order $2$.
\end{proposition}
\begin{proof}
Watkins's conjecture is known for elliptic curves of conductor $N<10000$. In particular, this includes all the elliptic curves with conductor a power of $2$. Then, we assume that the conductor is odd. By \eqref{rank}, $\rank_\Z (E(\Q))\leq 1$. Thus, it is enough to prove that if $m_E$ is odd, $\rank_\Z (E(\Q))=0$. \cite[Theorem 2.2]{agashe2012modular} says that if a prime $p$ divides the ratio $\delta_E/m_E$, then $p^2\mid N$, consequently, $\delta_E$ is odd. Finally, \cite[Theorem 1.1]{kazalicki2018special}, as corrected in \cite{kazalicki2019corrigendum}, implies that $\rank_\Z (E(\Q))=0$.
\end{proof}

\begin{proof}[Proof of Theorem \ref{main theorem}]
Note that the quadratic twists of $E^{(D)}$ are $E$ itself, or quadratic twists of $E$. Because of the classification of Section \ref{results} it is enough to prove the theorem for elliptic curves with conductor $2^5$, $2^6$, $17$, $49$ and prime numbers of the form $u^2+64$ for some integer $u$. 

Suppose that $E$ has conductor $N=49$, $E^{(7)}$  has conductor $N^{(7)}=49$ and a rational point of order $2$. Notice that if $7\nmid D'$, $E^{(7D')}=(E^{(7)})^{(D')}$, so, we can assume that $7\nmid D$, and we can then use Proposition \ref{Petersson Norm}.(I), freely. 

To begin with, we assume that $N$ is odd and $E$ is different to $17.a4$. By Corollary \ref{m/c2} and  Section \ref{classification} we have that $\nu_2(m_E/c_E^2)\geq -1$. Applying Corollary \ref{odd discriminant} and Proposition \ref{Petersson Norm}.(I) to Lemma \ref{valuation modular}, \eqref{val} turns into
\[
\nu_2(m_{E^{(D)}})\geq -1+3\omega(D).
\]
In the case of $17.a4$, we have that $\nu_2(m_E/c_E^2)=-2$ and therefore applying Proposition \ref{Petersson Norm}.(II) and Corollary \ref{odd discriminant} to Lemma \ref{valuation modular} we obtain
\[
\nu_2(m_{E^{(D)}})\geq -2+4\omega(D).
\]
Meanwhile, Lemma \ref{rankbound} implies that $\rank_\Z(E(\Q))\leq 2\omega(D)+1$, hence in both cases Watkins's conjecture holds for $\omega(D)\geq 2$. 

\noindent Now, assume $D$ is the fundamental discriminant associated to a prime number. Proposition \ref{wconj prime} proves the case $D\mid N$. Given Lemma \ref{rankbound}, we only have to prove that $\nu_2(m_{E^{(D)}})\geq 3$, then Remark \ref{D=1(mod 4)} proves the case $D\equiv 1\text{ (mod }4)$. For $D$ such that $2\mid D$, Corollary \ref{odd discriminant} implies that $(1/6)\nu_2\left(\Delta_{E^{(D)}}/\Delta_E\right)\geq 2$, then for $E$ different that $17.a4$ we have $\nu_2(m_{E^{(D)}})\geq -1+3+2=4$. Finally, when $E$ is $17.a4$ we obtain $\nu_2(m_{E^{(D)}})\geq -2+4+2=4$.

Now, suppose that $E$ has conductor $2^5$ or $2^7$, and different to $32.a3$, in which case $\nu_2(m_E/c_E^2)\geq 0$. As a consequence, Proposition \ref{Petersson Norm}.(III) and Lemma \ref{even discriminant} applied to Lemma \ref{valuation modular} imply
\begin{equation*}\label{lower bound modular degree}
\nu_2(m_{E^{(D)}})\geq 3\omega(D)-3\delta_2(D),    
\end{equation*}
and Lemma \ref{rankbound} says that $\rank_\Z (E(\Q))\leq 2\omega(D)+1-2\delta_2(D)$, and therefore Watkins's conjecture holds for $\omega(D)\geq 1+\delta_2(D)$. Thus, the only missing case is $D=2$, which is covered by Proposition \ref{wconj prime}.

Finally, for $E$ equal to $32.a3$ we have $\nu_2(m_E/c_E^2)=-1$. In this situation, Lemma \ref{rankbound} tells us that 
\[
\rank_\Z (E(\Q))\leq 2\omega(D)-\delta_2(D).
\]
On the other hand, Applying again Proposition \ref{Petersson Norm}.(IV) and Lemma \ref{even discriminant} to Lemma \ref{valuation modular} we obtain
\[
\nu_2(m_{E^{(D)}})\geq -1+4\omega(D)-4\delta_2(D),
\]
and we notice that the inequality $2\omega(D)-\nu_2(D)\leq -1+4\omega(D)-4\delta_2(D)$ is equivalent to
\[
\omega(D)\geq \frac{1+3\delta_2(D)}{2}.
\]
Then, the only missing case is $D=2$, which again is covered by Proposition \ref{wconj prime}.
\end{proof}
\section{Congruence Number of $y^2=x^3-dx$}
\cite[Theorem 2.8]{yazdani2011modular} shows that the elliptic curves of the form $y^2=x^3-Dx$ have even congruence number, whenever $\omega(D)\geq 1$. The idea of this section is to give a lower bound for $\nu_2(\delta_E)$. 
First of all, let $p_1,\dots,p_m$ be a list of distinct odd primes with $m\geq 1$ and define $d=p_1\cdots p_m>1$. Let $D$ be a divisor of $d$. Now, consider the elliptic curves $E:y^2=x^3-dx$ and $E^{(D)}:y^2=x^3-dD^2x$. Finally, denote by $f$ and $f^{(D)}$ their associated Hecke newforms. Since $E^{(D)}$ is a quadratic twist by $D$ of $E$, then we have that 
\begin{equation}\label{product-formula ap}
a_q(f^{(D)})=\left(\frac{D}{q}\right)a_q(f),
\end{equation}
for every prime number $q$. Before proving Theorem \ref{2main theorem} we need the following two lemmas.
\begin{lemma}\label{lemma 1}
Let $n$ be a positive integer relatively prime to $d$ and $q_1^{\alpha_1}\cdots q_s^{\alpha_s}$ its prime factorization. Then $a_n(f^{(D)})=\gamma_n(D)a_n(f)$, where
\[
\gamma_n(D)=\left(\frac{D}{q_1}\right)^{\alpha_1}\cdots \left(\frac{D}{q_s}\right)^{\alpha_s}.
\]
\end{lemma}
\begin{proof}
Since $\gamma_{nm}(D)=\gamma_n(D)\gamma_m(D)$, it is enough to show this assertion for power of primes. We prove it by induction, taking into account $a_1(f)=1$ and $a_q(f)= \left(\frac{d}{p}\right)a_q(g)$, as follows
\begin{align}\label{equation}
a_{q^{n+1}}(f^{(D)})&=a_q(f^{(D)})a_{q^{n}}(f^{(D)})-pa_{q^{n-1}}(f^{(D)})\\
&=\left(\frac{D}{q}\right)a_q(f)\left(\frac{D}{q}\right)^{n}a_{q^{n}}(f)-p\left(\frac{D}{q}\right)^{n-1}a_{q^{n-1}}(f)\nonumber\\
&=\left(\frac{D}{q}\right)^{n+1}a_{q^{n+1}}(f)\nonumber,
\end{align}
which gives the desired result.
\end{proof}
\begin{lemma}\label{lemma 2}
Let $m$ be an odd integer and $q$ be a prime, such that $q\nmid d$. If $\left(\frac{d}{q}\right)=1$ then  $a_{q^m}(f)\equiv 0 \text{ (mod }2)$ and if $\left(\frac{d}{q}\right)=-1$ then $a_{q^m}(f)\equiv 0 \text{ (mod }4)$. 
\end{lemma}
\begin{proof}
We know that $E$ is a quartic twist of $E_1:y^2=x^3-x$ by $d$. There is a Grossencharacter $\chi$ of $\Q(i)$ (equally $\overline{\chi}$), associated with the CM curve
$E_1/\Q$, in the sense that the Hecke eigenform attached to $E_1$ is a theta series
for $\chi$, and $L(E_1/\Q, s) = L(\chi/\Q(i), s)$ (see \cite[II. Theorem 10.5]{silverman1994}). The conductor of $\chi$ is $\mathfrak{p}^3$, where $\mathfrak{p}=(1+i)$, and order $4$. By \cite[Section 3.2]{PacettiCremona}, $\chi\psi$, where $\psi=\left(\frac{\cdot}{d}\right)_4$
is likewise associated with $E_d$. 

First of all, if $q\equiv 3 \text{ (mod }4)$ we have $a_q(E)=0$ since $\#E(\F_q)=q+1$ by \cite[Exercise 2.33(a)]{silverman1994}. By induction on \eqref{equation} we obtain that $a_{q^m}(E)=0$ for $m$ any odd integer.

Finally, assume that $q\equiv 1 \text{ (mod }4)$. 
Since $E$ has a rational $2$-torsion point we have that for every prime $q$ such that $(q,2d)=1$, $a_{q}(f)\equiv 0 \text{ (mod }2)$. Moreover, if $\left(\frac{d}{q}\right)=-1$ then $(a_q(f)/2)^2+ (a_q(g)/2)^2=p$, so, $a_{q}(f)\equiv 0 \text{ (mod }4)$. By induction on \eqref{equation}, we get the desired result.
\end{proof}
\begin{proof}[Proof of Theorem \ref{2main theorem}]
We define $n_1=p_1^{\nu_{p_1}(n)}\cdots p_m^{\nu_{p_m}(n)}$ and $n_2=n/n_1$. Let us also define the cusp form 
$$
h=\sum_{\substack{D\mid d\\D\neq 1}}(-1)^{\omega(D)+1}f^{(D)}.
$$ 
By  \cite[Table I]{hadano1975conductor} $N=N^{(D)}$ for every divisor $D$ of $d$. Consequently, for every divisor $D$ of $d$, $f^{(D)}$ and $f$ are orthogonal with respect to the Petersson inner product, since they are different newforms in the same level. Therefore, $h$ and $f$ are also orthogonal. Hence, it is enough to prove that:
$$a_n(f)-a_n(h)=\sum_{D\mid d}(-1)^{\omega(D)}a_n(f^{(D)})\equiv 0\text{ (mod }2^{m+\epsilon}),$$
where $\epsilon=1$ if $m$ is even and $\epsilon=2$ if $m$ is odd. Assume that $n_2=q_1^{\alpha_1}\cdots q_s^{\alpha_s}$, then by Lemma \ref{lemma 1} we have $a_{n_2}(f^{(D)})=\gamma_{n_2}(D)a_{n_2}(f)$. We claim that 
\[
\sum_{D\mid d}(-1)^{\omega(D)}a_n(f^{(D)})=\begin{cases}2^{m}a_n(f)& \text{if $\gamma_{n_2}(p)=-1$ for all $p\mid d$}\\
0 & \text{otherwise.}\end{cases}
\]
Before proving the claim, notice that if $p\mid d$, we have $a_p(f)=0$ and by \eqref{product-formula ap} $a_p(f^{(D)})=0$. Therefore, by \eqref{equation} $a_{n_1}(f)=a_{n_1}(f^{(D)})$ for every divisor $D$ of $d$, and consequently 
$$\sum_{D\mid d}(-1)^{\omega(D)}a_n(f^{(D)}=a_{n_1}(f)\sum_{D\mid d}(-1)^{\omega(D)}a_{n_2}(f^{(D)}).$$

To begin with, assume that for some $p\mid d$, $\gamma_{n_2}(p)=-1$. Since $\gamma_n(DD')=\gamma_n(D)\gamma_n(D')$. we have
\begin{align}\label{equation 1}
\sum_{D\mid d}(-1)^{\omega(D)}a_n\left(f^{(D)}\right)&=a_{n_1}(f)\left(\sum_{\substack{D\mid d\\ p\nmid D}}(-1)^{\omega(D)}a_{n_2}(f^{(D)})+\sum_{\substack{D\mid d \nonumber\\ p\mid D}}(-1)^{\omega(D)-1}a_{n_2}(f^{(D/p)})\right)\\
&=2a_{n_1}(f)\sum_{D\mid (d/p)}a_{n_2}\left(f^{(D)}\right).
\end{align}
Without loss of generality, assume that $t$ is an integer such that $0\leq t\leq m$ and for $i\leq t$ we have $\gamma_{n_2}(p_i)=-1$ , and for $i> t$ we have $\gamma_{n_2}(p_i)=1$. Denote by $d_1=p_1\cdots p_t$, then applying equation \eqref{equation 1} recursively, if $0\leq t<m$ we have that
\begin{align*}
\sum_{D\mid d}(-1)^{\omega(D)}a_n\left(f^{(D)}\right)=2^ta_{n_1}(f)\sum_{D\mid (d/d_1)}a_{n_2}(f^{(D)})=2^ta_n(f)\sum_{D\mid (d/d_1)}(-1)^{\omega(D)}=0.
\end{align*}
Meanwhile if $t=m$, we obtain that $a_n(\sum_{D\mid d}(-1)^{\omega(D)}f^{(D)})=2^ma_n(f)$.

Finally, if $\gamma_{n_2}(p)=-1$ for all $p\mid d$, we have that $\gamma_{q_i}(p)^{\alpha_i}=-1$ for some $1\leq i\leq s$, in particular, $\alpha_i$ is odd, then in view of Lemma \ref{lemma 2} we obtain that $a_n(f)$ is even. Even better, if $\omega(d)$ is odd, $\gamma_{n_2}(d)=-1$, then there exists $1\leq i\leq s$, such that $\gamma_{q_i}(d)^{\alpha_i}=-1$, therefore $\alpha_i$ is odd and $\left(\frac{d}{q_i}\right)=\gamma_{q_i}(d)=-1$. Applying Lemma \ref{lemma 2} we obtain that $4\mid a_n(f)$, so
$$\sum_{D\mid d}(-1)^{\omega(D)}a_n\left(f^{(D)}\right)\equiv 0\text{ (mod }2^{m+\epsilon}),$$
which proves the desired result.
\end{proof}

\begin{proof}[Proof of Corollary \ref{corollary 2main}]
By Remark \ref{remark rank} $\rank_\Z(E(\Q))\leq 2$ when $E$ is equal to $y^2=x^3-px$ or $y^2=x^3-p^3x$ for $p$ a prime number. Each of these elliptic curve is the quadratic twist by $p$ of the other, whose conductor is $2^5p^2$ if $p\equiv 1\text{ (mod }4)$ or $2^6p^2$ if $p\equiv -1\text{ (mod }4)$ (see \cite[Table I]{hadano1975conductor}). 

Furthermore, Theorem \ref{2main theorem} says that $3\leq \nu_2(\delta_E)$ in both cases, which yields the desired result.
\end{proof}

Let us remark that the lower bound for the $2$-adic valuation of the congruence number by Theorem \ref{2main theorem}
$2\left\lfloor\frac{\omega(d)+1}{2}\right\rfloor+1$ exceeds half of the upper bound for the Mordell-Weil rank $2\omega(d)-\delta_2(d)$, as noted in Remark \ref{remark rank}. 

Using data from the LMFDB \cite{LMFDB} for $d=21$ and $n\leq 100$, we have $\sum_{D\mid d}(-1)^{\omega(D)}a_n(f^{(D)})\equiv 16 \text{ (mod }32)$. On the other hand, Theorem \ref{2main theorem} implies $\delta(E)\geq 3$ in this scenario. This observation suggests that the actual congruence between $f$ and $h$ may not manifest at as high a power of $2$ as predicted by this theorem.

\end{document}